\documentclass[12pt]{article}
\usepackage{amsmath,amssymb,amsthm}
\usepackage{eucal}

\newcommand \comment[1]{}               
\newcommand \mylabel[1]{\label{#1}\comment{{\rm \{#1\} }}}

\newtheorem{lem}{Lemma}[section]
\newtheorem{cor}[lem]{Corollary}
\newtheorem{prop}[lem]{Proposition}
\newtheorem{thm}[lem]{Theorem}
\newtheorem{exama}[lem]{Example}
\newenvironment{example}{\begin{exama}\rm}{\end{exama}}

\numberwithin{equation}{section}
\allowdisplaybreaks

\renewcommand{\phi}{\varphi}                 
\renewcommand{\epsilon}{\varepsilon}
\newcommand\eset{\varnothing}

\newcommand\inv{^{-1}}
\newcommand\textb{\text{\rm b}}

\newcommand \ind{{:}}
\renewcommand\pmod[1]{\,(\operatorname{mod}#1)}
\newcommand\cupdot {\mbox{\hspace{.15em}$\cup$\hspace{-.47em}$\cdot$\hspace{.4em}}}

\newcommand\bQ{\mathbf{Q}}

\newcommand\cA{\mathcal{A}}
\newcommand\cB{\mathcal{B}}
\newcommand\cC{\mathcal{C}}
\newcommand\cG{\mathcal{G}}
\newcommand\cF{\mathcal{F}}
\newcommand\cH{\mathcal{H}}
\newcommand\cL{\mathcal{L}}
\newcommand\cM{\mathcal{M}}
\newcommand\cP{\mathcal{P}}

\newcommand\bbR{\mathbb{R}}
\newcommand\bbZ{\mathbb{Z}}

\newcommand \fA{\mathfrak A}
\newcommand \fG{\mathfrak G}

\newcommand\setm{\backslash}

\newcommand\G{\Gamma}
\newcommand\chiZ{\chi^\bbZ}
\newcommand\chim{\chi^\mathrm{mod}}
\newcommand\chib{\chi^\textb}
\newcommand\Comp{\operatorname{Comp}}
\newcommand\CT{\mathbf{C}}
\newcommand\CTN{\mathbf{C_0}}

\begin{document}


\begin{center}

{\large\sc 
An Elementary Chromatic Reduction for Gain Graphs\\ and Special Hyperplane Arrangements%
\footnote{October 2, 2007.


}
}\\[20pt]

{\large
Pascal Berthom\'e\footnote{Work performed while at the Laboratoire de Recherche en Informatique, Universit\'e Paris-Sud.} \\
{\small Laboratoire d'Informatique Fondamentale d'Orl\'eans,\\[-0.4ex]
ENSI de Bourges, Universit\'e d'Orl\'eans\\[-0.4ex]
88, Bld Lahitolle\\[-0.4ex]
18020 Bourges Cedex, France\\[-0.4ex]
\texttt{pascal.berthome@ensi-bourges.fr}
}\\[10pt]

{\large
Raul Cordovil\footnote{Partially supported by FCT (Portugal) through the program POCTI.  Much of Cordovil's work was performed while visiting the LRI, Universit\'e Paris-Sud.} \\
\small Departamento de Matematica\\[-0.4ex]
\small Instituto Superior T\'ecnico\\[-0.4ex]
\small Av.\ Rovisco Pais\\[-0.4ex]
\small 1049-001 Lisboa, Portugal\\[-0.4ex]
\small \texttt{rcordov@math.ist.utl.pt}
}\\[10pt]

{\large
David Forge \\
\small Laboratoire de Recherche en Informatique, UMR 8623\\[-0.4ex]
\small B\^at.\ 490, Universit\'e Paris-Sud\\[-0.4ex]
\small 91405 Orsay Cedex, France\\[-0.4ex]
\small \texttt{forge@lri.fr}
}\\[10pt]

{\large
V\'eronique Ventos \\
\small INRIA Futurs-Projet GEMO,\\[-0.4ex] 
\small Parc Club Orsay Universit\'e\\[-0.4ex] 
\small 4, rue Jacques Monod -- B\^at.\ G \\[-0.4ex]
\small 91405 Orsay Cedex, France \\[-0.4ex]
\small \texttt{ventos@lri.fr}
}\\[10pt]

{\large
Thomas Zaslavsky\footnote{Much of Zaslavsky's work was performed while visiting the LRI, Universit\'e Paris-Sud.} \\
\small Department of Mathematical Sciences \\[-0.4ex]
\small Binghamton University (SUNY)\\[-0.4ex]
\small Binghamton, NY 13902-6000, U.S.A.\\[-0.4ex]
\small \texttt{zaslav@math.binghamton.edu}
}\\[10pt]
}
\end{center}

\newpage

\begin{abstract}
A gain graph is a graph whose edges are labelled invertibly by \emph{gains} from a group.  \emph{Switching} is a transformation of gain graphs that generalizes conjugation in a group.  
A \emph{weak chromatic function} of gain graphs with gains in a fixed group satisfies three laws: deletion-contraction for links with neutral gain, invariance under switching, and nullity on graphs with a neutral loop.  
The laws are analogous to those of the chromatic polynomial of an ordinary graph, though they are different from those usually assumed of gain graphs or matroids.  The three laws lead to the \emph{weak chromatic group} of gain graphs, which is the universal domain for weak chromatic functions.  
We find expressions, valid in that group, for a gain graph in terms of minors without neutral-gain edges, or with added complete neutral-gain subgraphs, that generalize the expression of an ordinary chromatic polynomial in terms of monomials or falling factorials.  
These expressions imply relations for all switching-invariant functions of gain graphs, such as chromatic polynomials, that satisfy the deletion-contraction identity for neutral links and are zero on graphs with neutral loops.
Examples are the total chromatic polynomial of any gain graph, including its specialization the zero-free chromatic polynomial, and the integral and modular chromatic functions of an integral gain graph.  

We apply our relations to some special integral gain graphs including those that correspond to the Shi, Linial, and Catalan arrangements, thereby obtaining new evaluations of and new ways to calculate the zero-free chromatic polynomial and the integral and modular chromatic functions of these gain graphs, hence the characteristic polynomials and hypercubical lattice-point counting functions of the arrangements.  The proof involves gain graphs between the Catalan and Shi graphs whose polynomials are expressed in terms of descending-path vertex partitions of the graph of $(-1)$-gain edges.

We also calculate the total chromatic polynomial of any gain graph and especially of the Catalan, Shi, and Linial gain graphs.
\end{abstract}

\section{Introduction}

To calculate the chromatic polynomial $\chi_\G(q)$ of a simple graph there is a standard method that comes in two forms.  One can delete and contract edges, repeatedly applying the identity $\chi_\G = \chi_{\G\setm e} - \chi_{\G/e}$ and the reduction $\chi_\G(q)=0$ if $\G$ has a loop, to reduce the number of edges to zero.  One ends up with a weighted sum of chromatic polynomials of edgeless graphs, i.e., of monomials $q^k$.  Or, one can add missing edges using the opposite identity, $\chi_\G = \chi_{\G \cup e} + \chi_{\G/e}$, until $\chi_\G$ becomes a sum of polynomials of complete graphs, i.e., of falling factorials $(q)_k$.  We extend these approaches to gain graphs, where the edges are orientably labelled by elements of a group.  The resulting formulas are slightly more complex than those for simple graphs, but they can be used to compute examples; we show this with gain graphs related to the Shi, Linial, and Catalan hyperplane arrangements.

In a \emph{gain graph}, having edges labelled orientably means that reversing the direction of an edge inverts the label (the \emph{gain} of the edge).  Gain graphs, like ordinary graphs, have chromatic polynomials, which for various choices of gain group give the characteristic polynomials of interesting arrangements of hyperplanes.  In particular, when the gain group is $\bbZ$, the additive group of integers, the edges correspond to \emph{integral affinographic hyperplanes}, that is, hyperplanes of the form $x_j=x_i+m$ for integers $m$.  Arrangements (finite sets) of hyperplanes of this type, which include the Shi arrangement, the Linial arrangement, and the Catalan arrangement, have gained much interest in recent years.  The fact that the chromatic polynomials of gain graphs satisfy the classic deletion-contraction reduction formula has important consequences, e.g., a closed-form formula, and also a method of computation that can significantly simplify computing the polynomials of the Shi, Linial, and Catalan arrangements.

The deletion-contraction relation for functions can be viewed abstractly at the level of a Tutte group, which means that we take the free abelian group $\bbZ\cG(\fG)$ generated by all gain graphs with fixed gain group $\fG$, and from deletion-contraction we infer algebraic relations satisfied in a quotient of $\bbZ\cG(\fG)$.  This gives a group we call the \emph{neutral chromatic group} of gain graphs; it is a special type of Tutte group.  In particular, the relations reduce the number of generators.  These relations on graphs then, by functional duality, automatically generate the original deletion-contraction relations on chromatic polynomials and other functions of a similar type that we call \emph{weak chromatic functions} or, if they are invariant under switching (defined in the next section), simplification, and isomorphism, \emph{weak chromatic invariants}.

Our investigation in \cite{SOA} of the number of integer lattice points in a hypercube that avoid all the hyperplanes of an affinographic hyperplane arrangement led us to functions of integral gain graphs, the integral and modular chromatic functions, that count proper colorations of the gain graph from the color set $\{1,\ldots,q\}$, treated as either integers or modulo $q$.  
Like the chromatic polynomial, they satisfy a deletion-contraction identity, but (in the case of the integral chromatic function) only for links with neutral gain.  
That fact is part of what suggested our approach, and evaluating these chromatic functions for the Catalan, Shi, and Linial arrangements is part of our main results.

A brief outline:  The first half of the paper develops the general theory of functions on gain graphs that satisfy deletion-contraction for neutral links and are zero on gain graphs with neutral loops, in terms of  the neutral chromatic group.  It also develops, first of all, the corresponding theory for ordinary graphs, since that is one way we prove the gain-graphic reduction formulas.  The second half applies the theory to the computation of chromatic functions of the Catalan, Shi, and Linial gain graphs and a family of graphs intermediate between the Catalan and Shi graphs.  The latter can be computed in terms of partitions into descending paths of the vertex set of a graph.  This half also shows how to compute the total chromatic polynomial in terms of the zero-free chromatic polynomial; this in particular gives the chromatic polynomials of the Catalan, Shi, and Linial (and intermediate) graphs, although since the results are not as interesting as the method we do not state them.

\section{Basic definitions}\mylabel{defs}

For a nonnegative integer $k$, $[k]$ denotes the set $\{1,2,\dots,k\}$, the empty set if $k=0$.  The set of all partitions of $[n]$ is $\Pi_n$.  The falling factorial is $(x)_m=x(x-1)\cdots(x-m+1)$.

Suppose $\cP(E)$ is the power set of a finite set $E$, $S \mapsto \bar S$ is a closure operator on $E$, $\cF$ is the class of closed sets, and $\mu$ is the M\"obius function of $\cF$.  If the empty set is not closed, then $\mu(\eset,A)$ is defined to be $0$ for all $A \in \cF$.  It is known that:

\begin{lem}\mylabel{L:mu}
For each closed set $A$, $\mu(\eset,A) = \sum_{S} (-1)^{|S|}$, summed over edge sets $S$ whose closure is $A$.
\end{lem}

\begin{proof}
If $\eset$ is closed, this is \cite[Prop.\ 4.29]{Aigner}.  
Otherwise, $\sum_{\bar S=A} (-1)^{|S|} = \sum_{S\subseteq\bar\eset} (-1)^{|S|} x = 0x = 0$, where $x$ is a sum over some subsets of $E \setm \bar\eset$.
\end{proof}

Our usual name for a graph is $\G=(V,E)$.  Its vertex set is $V = \{v_1,v_2,\ldots,v_n\}$.  All our graphs are finite.  Edges in a graph are of two kinds: a \emph{link} has two distinct endpoints; a \emph{loop} has two coinciding endpoints.  Multiple edges are permitted.  
Edges that have the same endpoints are called \emph{parallel}.  
The \emph{simplification} of a graph is obtained by removing all but one of each set of parallel edges, including parallel loops.  (This differs from the usual definition, in which loops are deleted also.)    
The complement of a graph $\G$ is written $\G^c$; this is the simple graph whose adjacencies are complementary to those of $\G$.  
The complete graph with $W$ as its vertex set is $K_W$.  
For a partition $\pi$ of $V$, $K_\pi$ denotes a complete graph whose vertex set is $\pi$.  

If $S \subseteq E$, we denote by $c(S)$ the number of connected components of the spanning subgraph $(V,S)$ (which we usually simply call the ``components of $S$'') and by $\pi(S)$ the partition of $V$ into the vertex sets of the various components.  The complement of $S$ is $S^c = E \setm S$.  

An edge set $S$ in $\G$ is \emph{closed} if every edge whose endpoints are joined by a path in $S$ is itself in $S$.  $\cF(\G)$ denotes the lattice of closed sets of $\G$.   

Contracting $\G$ by a partition $\pi$ of $V$ means replacing the vertex set by $\pi$ and changing the endpoints $v, w$ of an edge $e$ to the blocks $B_v, B_w \in \pi$ that contain $v$ and $w$ (they may be equal); we write $\G/\pi$ for the resulting graph.

A vertex set is \emph{stable} if no edge has both endpoints in it.  
A \emph{stable partition} of $\G$ is a partition of $V$ into stable sets; let $\Pi^*(\G)$ be the set of all such partitions.  
Contracting a graph by a stable partition $\pi$ means collapsing each block of $\pi$ to a vertex and then simplifying parallel edges; there will be no loops. 

\medskip

A \emph{gain graph} $\Phi = (\G,\phi)$ consists of an underlying graph $\G$ and a function $\phi: E \to \fG$ (where $\fG$ is a group), which is orientable, so that if $e$ denotes an edge oriented in one direction and $e\inv$ the same edge with the opposite orientation, then $\phi(e\inv) = \phi(e)\inv$.  The group $\fG$ is called the \emph{gain group} and $\phi$ is called the \emph{gain mapping}.  A \emph{neutral edge} is an edge whose gain is the neutral element of the group, that is, $1_\fG$, or for additive groups $0$.  
The \emph{neutral subgraph} of $\Phi$ is the subgraph $\G_0 := (V,E_0)$ where $E_0$ is the set of neutral edges.  

We sometimes use the simplified notations $e_{ij}$ for an edge with endpoints $v_i$ and $v_j$, oriented from $v_i$ to $v_j$, and $ge_{ij}$ for such an edge with gain $g$; that is, $\phi(ge_{ij})=g$.  (Thus $ge_{ij}$ is the same edge as $g\inv e_{ji}$.)

A second way to describe a gain graph, equivalent to the definition, is as an ordinary graph $\G$, having a set of gains for each oriented edge $e_{ij}$, in such a way that the gains of $e_{ji}$ are the inverses of those of $e_{ij}$.  For instance, $K_n$ with additive gains $1,-1$ on every edge is a gain graph that has edges $1e_{ij}$, and $-1e_{ij}$ for every $i \neq j$.  Since this is a gain graph, $1e_{ij}$ and $-1e_{ji}$ are the same edge.

Two gain graphs are \emph{isomorphic} if there is a graph isomorphism between them that preserves gains.  

The \emph{simplification} of a gain graph is obtained by removing all but one of each set of parallel edges, including parallel loops, that have the same gain.  (Unlike the usual definition, we do not mean to remove all loops.)  

A \emph{circle} is a connected 2-regular subgraph, or its edge set.  The gain of $C=e_1e_2\cdots e_l$ is $\phi(C):=\phi(e_1)\phi(e_2)\cdots\phi(e_l)$; this is not entirely well defined, but it is well defined whether the gain is or is not the neutral element of $\fG$.  An edge set or subgraph is called \emph{balanced} if every circle in it has neutral gain.

\emph{Switching} $\Phi$ by a \emph{switching function} $\eta : V \to \fG$ means replacing $\phi$ by $\phi^\eta$ defined by 
\[
\phi^\eta(e_{ij}):=\eta_i\inv\phi(e_{ij})\eta_j,
\]  
where $\eta_i := \eta(v_i)$.  
We write $\Phi^\eta$ for the switched gain graph $(\G,\phi^\eta)$.  Switching does not change balance of any subgraph.  

The operation of deleting an edge or a set of edges is obvious, and so is contraction of a neutral edge set $S$:  we identify each block $W\in \pi(S)$ to a single vertex and delete $S$ while retaining the gains of the remaining edges.  (This is just as with ordinary graphs.)  The contraction is written $\Phi/S$.  
A \emph{neutral-edge minor} of $\Phi$ is any graph obtained by deleting and contracting neutral edges; in particular, $\Phi$ is a neutral-edge minor of itself.

Contraction of a general balanced edge set is not so obvious.  Let $S$ be such a set.  There is a switching function $\eta$ such that $\phi^\eta\big|_S=1_\fG$ \cite[Section I.5]{BG}; and $\eta$ is determined by one value in each component of $S$.  If the endpoints of $e$ are joined by a path $P$ in $S$, $\phi^\eta(e)$ is well defined as the value $\phi(P\cup e)$, the gain of the circle formed by $P$ and $e$.  Now, to contract $S$ we first switch so that $S$ has all neutral gains; then we contract $S$ as a neutral edge set.  This contraction is also written $\Phi/S$.  
It is well defined only up to switching because of the arbitrary choice of $\eta$.  However, when contracting a neutral edge set we always choose $\eta \equiv 1_\fG$; then $\Phi/S$ is completely well defined.  

An \emph{edge minor} of $\Phi$ is a graph obtained by any sequence of deletions and contractions of any edges; thus, for instance, $\Phi$ is an edge minor of itself.

\emph{Contracting $\Phi$ by a partition} $\pi$ of $V$ means identifying each block of $\pi$ to a single vertex without changing the gain of any edge.  The notation for this contraction is $\Phi/\pi$.  (A contraction of $\Phi$ by a partition is not an edge minor.)

\medskip

If $W\subseteq V$, the subgraph induced by $W$ is written $\G\ind W$ or $\Phi\ind W$.  
For $S \subseteq E$, $S\ind W$ means the subset of $S$ induced by $W$.

\section{Chromatic relations on graphs}

We begin with weak chromatic functions of ordinary graphs.  This introduces the ideas in the relatively simple context of graphs; we also use the graph case in proofs.  
The models are the chromatic polynomial, $\chi_\G(q)$, and its normalized derivative the beta invariant $\beta(\G) := (-1)^n (d/dq)\chi_\G(1)$.  Regarded as a function $F$ of graphs, a weak chromatic function has three fundamental properties:  the \emph{deletion-contraction} law,
\begin{align*}
F(\G) &= F(\G\setm e) - F(\G/e) &&\text{ for every link $e$};
\intertext{\emph{loop nullity},}
F(\G) &= 0 &&\text{ if $\G$ has a loop};
\intertext{and \emph{invariance under simplification},}
F(\G') &= F(\G) &&\text{ if $\G'$ is obtained from $\G$ by simplification}.
\end{align*}
Two usually important properties of which we have no need are isomorphism invariance (with the obvious definition) and multiplicativity (the value of $F$ equals the product of its values on components; the chromatic polynomial is multiplicative but the beta invariant is not).

Let $\cG$ be the class of all graphs.  A function $F$ from $\cG$ to an abelian group is a \emph{weak chromatic function} if it satisfies deletion-contraction, invariance under simplification, and loop nullity.  
(Chromatic functions are a special type of Tutte function, as defined in \cite{STF}.  The lack of multiplicativity is why our functions are ``weak''; cf.\ \cite{STF}.)  
One can take a function $F$ with smaller domain, in particular, the set $\cM(\G_0)$ of all \emph{edge minors} of a fixed graph $\G_0$ (these are the graphs obtained from $\G_0$ by deleting and contracting edges), although then Lemma \ref{simplifygraph} is less strong.

The first result shows that loop nullity is what particularly distinguishes chromatic functions from other functions that satisfy the deletion-contraction property, such as the number of spanning trees.  

\begin{lem}\mylabel{simplifygraph}
Given a function $F$ from all graphs to an abelian group that satisfies deletion-contraction for all links, loop nullity is equivalent to invariance under simplification.

Let $\G_0$ be any graph.  Given a function $F$ from $\cM(\G_0)$ to an abelian group that satisfies deletion-contraction for all links, loop nullity implies invariance under simplification.
\end{lem}

\begin{proof}
Assume $F$ has loop nullity.  Suppose we contract an edge in a digon $\{e,f\}$.  The other edge becomes a loop.  By deletion-contraction and loop nullity, 
$F(\G) = F(\G\setm e) - F(\G/e) = F(\G\setm e) + 0.$

Conversely, assume $F$ is invariant under simplification.  If $\G'$ has a loop $f$, it is the contraction of a graph $\G$ with a digon $\{e,f\}$ and the preceding reasoning works in reverse.  This reasoning applies if the domain of $F$ contains a graph $\G$ with the requisite digon; that is certainly true if the domain is $\cG$.
\end{proof}

In other words, one may omit invariance under simplification from the definition of a weak chromatic function.  This is good to know because for some functions it is easier to establish loop nullity than invariance under simplification.

A weak chromatic function can be treated as a function of vertex-labelled graphs.  
(A \emph{vertex-labelled (simple) graph} has a vertex set and a list of adjacent vertices.)
To see this we need only observe that graphs with loops can be ignored, since they have value $F(\G)=0$ by definition, and for graphs without loops, the set of vertices and the list of adjacent pairs determines $F(\G)$.  To prove the latter, consider two simple graphs, $\G_1$ and $\G_2$, with the same vertex set and adjacencies.  Let $\G$ be their union (on the same vertex set).  Both $\G_1$ and $\G_2$ are simplifications of $\G$, so all three have the same value of $F$.  
When we regard $F$ as defined on vertex-labelled graphs, contraction is \emph{simplified contraction}; that is, parallel edges in a contraction $\G/e$ are automatically simplified.  

Now we introduce the algebraic formalism of weak chromatic functions.  
The \emph{chromatic group} for graphs, $\CT$, is the free abelian group $\bbZ\cG$ generated by all graphs, modulo the relations implied by deletion-contraction, invariance under simplification, and loop nullity.  These relations are: 
\begin{equation}\mylabel{E:graphdcnl}
\begin{aligned}
\G &= (\G \setm e) - (\G/e) &\text{ for a link $e$}, \\
\G' &= \G &\text{ if $\G'$ is obtained from $\G$ by simplification}, \\
\G &= 0 &\text{ if $\G$ has a loop}.
\end{aligned}
\end{equation}
The point of these relations is that any homomorphism from the chromatic group to an abelian group will be a weak chromatic function of graphs, and any such function of graphs that has values in an abelian group $\fA$ is the restriction to $\cG$ of a (unique) homomorphism from $\CT$ to $\fA$.  These facts follow automatically from the definition of the chromatic group.  
Thus, $\CT$ is the universal abelian group for weak chromatic functions of graphs.  

It follows from the definition (the proof is similar to the preceding one for functions) that two graphs are equal in $\CT$ if they simplify to the same vertex-labelled graph; that is, we may treat $\CT$ as if it were generated by vertex-labelled graphs and contraction as simplified contraction.

We may replace $\cG$ by $\cM(\G)$; $\CT(\G)$ denotes the corresponding chromatic group, i.e., $\bbZ\cM(\G)$ modulo the relations \eqref{E:graphdcnl}.  Then $\CT(\G)$ is the universal abelian group for weak chromatic functions defined on the edge minors of $\G$.

Now we present the main result about graphs.  Let $\cF(\G_0)$ be the lattice of closed edge sets of the ordinary graph $\G_0$, and write $\mu$ for its M\"obius function.  If the empty set is not closed (that is, if there are loops in $\G_0$), then $\mu(\eset,S)$ is defined to be identically $0$.  

\begin{lem}\mylabel{graph}
In the chromatic group $\CT(\G_0)$ of a graph $\G_0$, for any edge minor\/ $\G$ of\/ $\G_0$ we have the relations
\[
\G = \sum_{S\in \cF(\G)}  \mu(\eset,S) [(\G/S) \setm S^c] = \sum_{S\subseteq E}  (-1)^{|S|} [(\G/S) \setm S^c]
\]
and
\[
\G = \sum_{\pi\in\Pi^*(\G)} K_\pi .\]
\end{lem}

\begin{proof}
The two sums in the first identity are equal by Lemma \ref{L:mu} and because $(\G/S) \setm E = (\G/R) \setm E$.  

We prove the first identity for graphs $\G$ without loops by induction on the number of edges in $\G$.  Let $e$ be a link in $\G$.
\begin{align*}
\G &= (\G\setm e) - (\G/e) \\
 &= \sum_{S \subseteq E\setm e}  (-1)^{|S|} [(\G/S) \setm S^c] - \sum_{A \subseteq E\setm e}  (-1)^{|A|} [([\G/e]/A) \setm A^c] \\
 &= \sum_{e \notin S \subseteq E}  (-1)^{|S|} [(\G/S) \setm S^c] - \sum_{e \in S \subseteq E}  (-1)^{|S|-1} [(\G/S) \setm S^c] \\
 &= \sum_{S\subseteq E}  (-1)^{|S|} [(\G/S) \setm S^c] ,
\end{align*}
where in the middle step we replaced $A \subseteq E\setm e$ by $S = A \cup e$.

If $\G$ has loops, let $e$ be a loop.  Then $\G=0$; at the same time 
\begin{align*}
\sum_{S\subseteq E}  (-1)^{|S|} [(\G/S) \setm S^c] 
&= \sum_{e \in S\subseteq E}  (-1)^{|S|} [(\G/S) \setm S^c] + \sum_{e \notin S\subseteq E}  (-1)^{|S|} [(\G/S) \setm S^c] \\
&= \sum_{e \notin S\subseteq E}  \Big[ (-1)^{|S|+1} [(\G/S/e) \setm (S\setm e)^c] +  (-1)^{|S|} [(\G/S) \setm S^c] \Big] ,
\end{align*}
which equals $0$ because for a loop, $\G/S/e = \G/S \setm e$.

We prove the second identity by induction on the number of edges not in $\G$.  Each stable partition of $\G \setm e$ is either a stable partition of $\G$, or has a block that contains the endpoints of $e$.  In that case, contracting $e$ gives a stable partition of $\G/e$.  Thus,
\[
\G \setm e = \G + (\G/e) = \sum_{\pi\in\Pi^*(\G)} K_\pi + \sum_{\pi\in\Pi^*(\G/e)} K_\pi = \sum_{\pi\in\Pi^*(\G \setm e)} K_\pi .
\qquad\qedhere
\]
\end{proof}

Validity of the identities in $\CT(\G_0)$ implies they are valid in $\CT$, since the former maps homomorphically into the latter by linear extension of the natural embedding $\cM(\G_0) \to \cG$.  
(We cannot say that this homomorphism of chromatic groups is injective.  The relations in $\CT$ might conceivably imply relations amongst the minors of $\G_0$ that are not implied by the defining relations of $\CT(\G_0)$.  We leave the question of injectivity open since it is not relevant to our work.)

\section{Neutral chromatic functions and relations on gain graphs}

We are interested in functions on gain graphs, with values in some fixed abelian group, that satisfy close analogs of the chromatic laws for graphs, which in view of Lemma \ref{simplifygraph} are two: deletion-contraction and loop nullity.  
The \emph{neutral deletion-contraction} relation is the deletion-contraction identity
\begin{equation}\mylabel{E:dc}
F(\Phi) = F(\Phi\setm e) - F(\Phi/e)
\end{equation}
for the special case where $e$ is a neutral link.  \emph{Neutral-loop nullity} is the identity
\[
F(\Phi) = 0 \text{ if $\Phi$ has a neutral loop}.
\]
Neutral deletion-contraction is a limited version of a property that in the literature is usually required (if at all) of all or nearly all links.  
We call a function that adheres to these two properties a \emph{weak neutral chromatic function} of gain graphs; ``weak'' because it need not be multiplicative, ``neutral'' because only neutral edges must obey the two laws.  
(To readers familiar with the half and loose edges of \cite{BG}: a loose edge is treated like a neutral loop.)

A function is \emph{invariant under neutral-edge simplification} if
\begin{enumerate}
\item[] $F(\Phi) = F(\Phi')$ when $\Phi'$ is obtained from $\Phi$ by removing one edge of a neutral digon.
\end{enumerate}

\begin{lem}\mylabel{simplifyneutral}
Given a function $F$ of all gain graphs with a fixed gain group that satisfies deletion-contraction for all neutral links, neutral-loop nullity is equivalent to invariance under neutral-edge simplification.
\end{lem}

The proof is like that of Lemma \ref{simplifygraph} so we omit it.

The \emph{neutral chromatic group} for $\fG$-gain graphs, $\CTN(\fG)$, is the free abelian group $\bbZ\cG(\fG)$ generated by the class $\cG(\fG)$ of all gain graphs with the gain group $\fG$, modulo the relations implied by deletion-contraction of neutral links and neutral-loop nullity.  These relations are
\begin{equation}\mylabel{E:dcnl}
\begin{aligned}
\Phi &= (\Phi\setm e) - (\Phi/e) \text{ for a neutral link $e$}, \\
\Phi &= 0 \text{ if $\Phi$ has a neutral loop}.
\end{aligned}
\end{equation}
As with graphs, the purpose of these relations is that any homomorphism from the neutral chromatic group to an abelian group will be a function of $\fG$-gain graphs that satisfies neutral deletion-contraction and neutral-loop nullity, and every function of $\fG$-gain graphs that satisfies those two properties, and has values in an abelian group $\fA$, is the restriction of a (unique) homomorphism from $\CTN(\fG)$ to $\fA$.  (These facts follow automatically from the definition of the neutral chromatic group.)  
Thus, $\CTN(\fG)$ is the universal abelian group for functions satisfying the two properties.

In the neutral chromatic group we get relations between gain graphs, in effect, by deleting and contracting neutral edges to expand any gain graph in terms of gain graphs with no neutral edge, while by addition and contraction we express it in terms of gain graphs whose neutral subgraph is the spanning complete graph $1_\fG K_n$.

Recall that $\cF(\G_0)$ is the lattice of closed sets of $\G_0$ and $\Pi^*(\G_0)$ is the set of stable partitions.  Let $\mu_0$ be the M\"obius function of $\cF(\G_0)$.  Recall also that, for a gain graph $\Phi$, $\G_0$ is the neutral subgraph of $\Phi$ and $E_0$ is the edge set of $\G_0$.

\begin{thm}\mylabel{first}
In the neutral chromatic group $\CTN(\fG)$ we have the relation
\[
\Phi = \sum_{S\in \cF(\G_0)}  \mu_0(0,S) [(\Phi/S) \setm E_0] = \sum_{S\subseteq E_0}  (-1)^{|S|} [(\Phi/S) \setm E_0] .
\]
\end{thm}

\begin{thm}\mylabel{second}
In the neutral chromatic group $\CTN(\fG)$ we have the relation
\[
\Phi = \sum_{\pi\in\Pi^*(\G_0)} [(\Phi/\pi) \cup 1_\fG K_{\pi}] .
\]
\end{thm}

One can easily give direct proofs of Theorems \ref{first} and \ref{second} just like those of the two identities in Lemma \ref{graph}.  We omit these proofs in favor of ones that show the theorems are natural consequences of the relations for ordinary graphs; thus we deduce them by applying a homomorphism to the relations in Lemma \ref{graph}.

\begin{proof}[Homomorphic Proof of Theorem \ref{first}.]
All vertices and edges are labelled, i.e., identified by distinct names.  
The vertices of a contraction are labelled in a particular way: the contraction by an edge set $S$ has vertex set $\pi(S)$, the partition of $V$ into the vertex sets of the connected components of $(V,S)$, and its edge set is the complement $S^c$ of $S$.  If we contract twice, say by (disjoint) subsets $S$ and $S'$, then we label the vertices of the contraction as if $S\cup S'$ had been contracted in one step.  

Now, define a function $f : \cM(\G_0) \to \cG(\fG)$ by
\[
f(\G_0/S \setm T) := \Phi/S \setm T .
\]
Given an edge minor $\G_0/S\setm T$ of the neutral subgraph $\G_0$, even though we cannot reconstruct $S$ and $T$ separately, we can reconstruct the vertex partition $\pi(S)$ by looking at the labels of the vertices of the minor, and we can reconstruct $S \cup T$ by looking at the surviving edges of the minor.  It follows that $f$ is well defined, because its value on a minor of $\G_0$ does not depend on which edges are contracted and which are deleted, as long as the vertex partition and surviving edge set are the same.  
(One can write this fact as a formula:  $\G_0/S\setm T = \G_0/\pi(S) \setm (S\cup T)$.)  

Extend $f$ linearly to a function $\bbZ\cM(\G_0) \to \bbZ\cG(\fG)$ and define $\bar f: \bbZ\cM(\G_0) \to \CTN(\fG)$ by composing with the canonical mapping $\bbZ\cG(\fG) \to \CTN(\fG)$.  
The kernel of $\bar f$ contains all the expressions $G - [(G\setm e) - (G/e)]$ for links $e$ of edge minors $G \in \cM(\G_0)$ and all expressions $G$ for edge minors with loops, because $\bar f$ maps them all to $0 \in \CTN(\fG)$ due to \eqref{E:dcnl}.  
Therefore, $\bar f$ induces a homomorphism $F: \CT(\G_0) \to \CTN(\fG)$.

Applying $F$ to the first formula of Lemma \ref{graph}, we get
\begin{equation*}
F(\G_0) = \sum_{S\in \cF(\G_0)}  \mu_0(0,S) F((\Phi/S) \setm E_0) .
\end{equation*}
This is the theorem.
\end{proof}

\begin{proof}[Homomorphic Proof of Theorem \ref{second}.]
Apply the same $F$ to the second formula of Lemma \ref{graph}.
\end{proof}

If $\Phi$ is a gain graph, let $\Phi^0 := \Phi \cup 1_\fG K_n$, i.e., $\Phi$ with all possible neutral links added in.  

\begin{cor}\mylabel{complete}
In the neutral chromatic group, if $\Phi$ has no neutral edges, then
\[
\Phi^0 = \sum_{\pi\in\Pi_n}  \mu(0,\pi) (\Phi/\pi) ,
\]
where $\mu$ is the M\"obius function of $\Pi_n$, and
\[
\Phi = \sum_{\pi\in\Pi_n} (\Phi/\pi)^0 .
\]
\end{cor}

\begin{proof}
Indeed, the identities follow from Theorems \ref{first} and \ref{second}.  
In the theorems the graph of neutral edges of $\Phi^0$ is the complete graph.  So the flats are exactly the partitions of $[n]$.  Contracting $\Phi$ by $\pi$ introduces no neutral edges so in Theorem \ref{first} it is not necessary to delete them.
\end{proof}

\section{Weak chromatic invariants of gain graphs}\mylabel{wci}

One can strengthen the definition of a weak chromatic function by requiring it to be invariant under some transformation of the gain graph.  Examples:
\begin{itemize}
\item \emph{Isomorphism invariance}:  The value of $F$ is the same for isomorphic gain graphs.
\item \emph{Switching invariance}:   The value of $F$ is not changed by switching.
\item \emph{Invariance under simplification}:  $F$ takes the same value on a gain graph and its simplification.
\end{itemize}
A \emph{weak chromatic invariant} of gain graphs is a function that satisfies the deletion-contrac\-tion formula \eqref{E:dc} for all links, neutral-loop nullity, and invariance under isomorphism, switching, and simplification.  It is a \emph{chromatic invariant} if it is also multiplicative on connected components, i.e., 
\begin{equation}\mylabel{E:mult}
F(\Phi_1 \cupdot \Phi_2) = F(\Phi_1) F(\Phi_2).
\end{equation}

\begin{lem}\mylabel{L:simplify}
Let $F$ be a function on all gain graphs with a fixed gain group that is switching invariant and satisfies deletion-contraction for neutral links.  Then neutral-loop nullity is equivalent to invariance under simplification and it implies isomorphism invariance.
\end{lem}

\begin{proof}
Suppose $F$ is switching invariant and satisfies deletion-contraction for all neutral links.  Then it satisfies deletion-contraction for all links, because any gain graph $\Phi$ can be switched to give gain $1_\fG$ to any desired link.

Suppose $F$ also has neutral-loop nullity.  When we contract an edge $e$ in a balanced digon $\{e,f\}$, $f$ becomes a neutral loop because of the switching that precedes contraction.  
By deletion-contraction and neutral-loop nullity,
\[
F(\Phi) = F(\Phi\setm e) - F(\Phi/e) = F(\Phi\setm e) + 0.
\]
Conversely, if $\Phi'$ has a neutral loop $f$, it is the contraction of a gain graph $\Phi$ with a neutral digon $\{e,f\}$ and the same reasoning works in reverse.

Now we deduce isomorphism invariance from invariance under simplification.  Suppose we have two isomorphic gain graphs, $\Phi$ and $\Phi'$, with different edge sets.  We may assume the edge sets are disjoint and that, under the isomorphism, the vertex bijection is the identity and $e \leftrightarrow e'$ for edges.  Take the union of the two graphs on the same vertex set; the union contains balanced digons $\{e,e'\}$.  If we remove $e$ from each pair we get $\Phi'$, but if we remove $e'$ we get $\Phi$.  The value of $F$ is the same either way.

This treatment omits loops.  Balanced loops make $F$ equal to $0$.  For unbalanced loops we induct on their number.  We can treat $\Phi$ as the contraction $\Psi/f$ where $\Psi$ is a gain graph in which $e,f$ are parallel links and $f$ is neutral, and similarly $\Phi' = \Psi'/f'$.  Since $F(\Psi) = F(\Psi')$ and $F(\Psi\setm e) = F(\Psi'\setm e')$ by induction, $F(\Phi) = F(\Phi')$ by deletion-contraction.
\end{proof}

\begin{prop}\mylabel{invariance}
A switching-invariant weak chromatic function of all gain graphs with a fixed gain group is a weak chromatic invariant.
\end{prop}

\begin{proof}
By switching we can make any link into a neutral link.  Apply Lemma \ref{L:simplify}.
\end{proof}

(We could formulate these properties of functions in terms of new chromatic groups, which are quotients of the weak chromatic group obtained by identifying gain graphs that are equivalent under a suitable equivalence, like simplification, switching, or isomorphism.  However, that would contribute nothing to our general theory and it seems an overly complicated way to do the computations in the second half of this paper.)

\begin{example}\mylabel{X:tcpoly}
Weak chromatic invariants abound, but the most important is surely the total chromatic polynomial.  A \emph{multi-zero coloration} is a mapping $\kappa: V \to (\fG\times[k]) \cup [z]$, where $k$ and $z$ are nonnegative integers.  It is \emph{proper} if it satisfies none of the following edge constraints, for any edge $e_{ij}$: 
\begin{align*}
\kappa(v_j) &= \kappa(v_i) \in [z]
\intertext{or}
\kappa(v_i) &= (m,g) \in \fG\times[k] \text{ and } \kappa(v_j) = (m,g\phi(e_{ij})) .
\end{align*}
When $\fG$ is finite, the \emph{total chromatic polynomial} is defined by 
\begin{equation}\mylabel{E:tcpolyc}
\tilde\chi_\Phi(q,z) = \text{ the number of proper multi-zero colorations},
\end{equation}
where $q = k|\fG| + z$.  This function combines the \emph{chromatic polynomial},
\begin{equation}\mylabel{E:cp}
\chi_\Phi(q) = \tilde\chi_\Phi(q,1),
\end{equation}
and the \emph{zero-free chromatic polynomial},
\begin{equation}\mylabel{E:zcp}
\chi^*_\Phi(q) = \tilde\chi_\Phi(q,0),
\end{equation}
of \cite{SGC} and \cite[Part III]{BG}.  
All three polynomials generalize the chromatic polynomial, for, regarding an ordinary graph $\Gamma$ as a gain graph with gains in the trivial group, we see that 
\[
\tilde\chi_\Gamma(q,z) = \chi_\Gamma(q)
\]
(which is independent of $z$), where $\chi_\Gamma(q)$ is the usual chromatic polynomial of $\Gamma$.

A second definition of the chromatic polynomials, which is algebraic, applies to all gain graphs, including those with infinite gain group.  We define a total chromatic polynomial for any gain graph by the formula 
\begin{equation}\mylabel{E:tcpoly}
\tilde\chi_\Phi(q,z) := \sum_{S\subseteq E} (-1)^{|S|} q^{b(S)} z^{c(S)-b(S)} ,
\end{equation}
where $b(S)$ is the number of components of $(V,S)$ that are balanced, and we define the chromatic and zero-free chromatic polynomials by means of \eqref{E:cp} and \eqref{E:zcp}.

\begin{prop}\mylabel{P:tcpoly}
The total chromatic polynomial is a weak chromatic invariant of gain graphs.  The combinatorial and algebraic definitions, \eqref{E:tcpolyc} and \eqref{E:tcpoly}, agree when both are defined.  
\end{prop}

\begin{proof}
The first task is to show that the two definitions of the total chromatic polynomial agree.  
The combinatorial total chromatic polynomial is the special case of the state chromatic function $\chi_\Phi(\bQ)$ of \cite[Section 2.2]{TFS} in which the spin set $\bQ = (\fG\times[k]) \cup [z]$.  That is, $\tilde\chi_\Phi(k|\fG|+z,z) = \chi_\Phi(\bQ)$.  This is obvious from comparing the definitions.  Indeed, $\tilde\chi_\Phi(q,z)$ as defined in \eqref{E:tcpolyc} is precisely the state chromatic function $\chi_{\Phi;\bQ_1,\bQ_2}(k_1,k_2)$ of the example in \cite[Section 4.3]{TFS} with the substitutions $q = k_1|\fG|+k_2$ and $z = k_2$.  

Consequently, the combinatorial total chromatic polynomial has all the properties of a state chromatic function.  The chief of these properties is that it agrees with the algebraic polynomial of \eqref{E:tcpoly}.  This fact is \cite[Equation (4.3)]{TFS} combined with Lemma \ref{L:mu} and the observation that the fundamental closure of $S$ has the same numbers of components and of balanced components as does $S$ \cite[p.\ 144]{TFS}.

The second task is to prove that the algebraic total chromatic polynomial is a chromatic invariant.  
Isomorphism invariance is obvious from the defining equation \eqref{E:tcpoly}.  
Switching invariance follows from the fact that $b(S)$ and $c(S)$ are unchanged by switching.  
Multiplicativity, Equation \eqref{E:mult}, is easy to prove by the standard method of splitting the sum over $S$ into a double sum over $S \cap E(\Phi_1)$ and $S \cap E(\Phi_2)$.
Reasoning like that in the proof of Lemma \ref{graph} proves neutral-loop nullity.  
By Proposition \ref{invariance}, if $\tilde\chi_\Phi$ satisfies deletion-contraction, \eqref{E:dc}, for every link then it is invariant under simplification and thus is a chromatic invariant.  

Thus, we must prove that $\tilde\chi_\Phi$ does satisfy deletion-contraction with respect to a link $e$.  The method is standard---e.g., see the proof of \cite[Theorem III.5.1]{BG}.  
We need two formulas about the contraction $\Phi/e$.  Suppose $e \in S \subseteq E$.  
Clearly, $c_{\Phi}(S) = c_{\Phi/e}(S \setm e)$.  \cite[Lemma I.4.3]{BG} tells us that $b_\Phi(S) = b_{\Phi/e}(S \setm e)$.  
Now we calculate: 
\begin{align*}
\tilde\chi_\Phi(q,z) - \tilde\chi_{\Phi\setm e}(q,z) 
&= \sum_{\substack{S\subseteq E \\ e \in S}}  (-1)^{|S|} q^{b_\Phi(S)} z^{c_\Phi(S)-b_\Phi(S)} \\
&= \sum_{T \subseteq E \setm e} (-1)^{|T|+1} q^{b_{\Phi/e}(T)} z^{c_{\Phi/e}(T)-b_{\Phi/e}(T)} \\
&= - \tilde\chi_{\Phi/e}(q,z) ,
\end{align*}
where again $T := S \setm e$.  This proves \eqref{E:dc}.  

By Proposition \ref{invariance}, therefore, $\tilde\chi_\Phi$ is a chromatic invariant of gain graphs.
\end{proof}
\end{example}

\section{Integral gain graphs and integral affinographic hyperplanes}\mylabel{igg}

An \emph{integral gain graph} is a gain graph whose gain group is the additive group of integers, $\bbZ$.  The ordering of the gain group $\bbZ$ singles out a particular switching function $\eta_S$: it is the one whose minimum value on each block of $\pi(S)$ is zero.  We call this the \emph{top switching function}.  The contraction rule is that one uses the top switching function; thus the contraction can be uniquely defined, unlike the situation in general.  

Contraction of a balanced edge set $S$ in an integral gain graph can be defined quite explicitly.  

First, we define $\eta_S$.  In each component $(V_i,S_i)$ of $S$, pick a vertex $w_i$ and, for $v \in V_i$, define $\eta(v) := \phi(S_{vw_i})$ for any path $S_{vw_i}$ from $v$ to $w_i$ in $S$.  ($\eta$ is well defined because $S$ is balanced.)  Let $v_i$ be a vertex which minimizes $\eta(v)$ in $V_i$.  Define $\eta_S(v) := \phi(S_{vv_i}) = \eta(v) - \eta(v_i)$ for $v \in V_i$.  Then $\eta_S$ is the top switching function for $S$, since $\eta_S(v_i) = 0 \leq \eta_S(v)$ for all $v \in V_i$.  

Next, we switch.  In $\Phi^{\eta_S}$, the gain of an edge $e_{vw}$, where $v \in V_i$ and $w \in V_j$, is $\phi^{\eta_S}(e_{vw}) = -\eta_S(v) + \phi(e_{vw}) + \eta_S(w) = \phi(S_{v_iv}) + \phi(e_{vw}) + \phi(S_{wv_j}) = \phi(S_{v_iv}e_{vw}S_{wv_j})$.  That is, $\phi^{\eta_S}(e_{vw})$ is the gain of a path from $v_i$ to $v_j$ that lies entirely in $S$ except for $e_{vw}$ if that edge is not in $S$.  (If $e_{vw}$ is in $S$, its switched gain is $0$, consistent with the fact that then $v_i=v_j$.)

Finally, we contract $S$.  We can think of this as collapsing all of $V_i$ into the single vertex $v_i$ and deleting the edges of $S$, while not changing the gain of any edge outside $S$ from its switched gain $\phi^{\eta_S}(e_{vw}) = \phi(S_{v_iv}e_{vw}S_{wv_j})$.  (If there happens to be an edge $v_iv_j$, it will have the same gain in $\Phi/S$ as it did in $\Phi$.)

\medskip

A kind of invariance that will now become important is:
\begin{itemize}
\item \emph{Loop independence}:  The value of $F$ is not changed by removing nonneutral loops.
\end{itemize}
Loop independence, when it holds true, permits calculations by means of contraction.  The zero-free chromatic polynomial and both of the next two examples have loop independence, which we employ to good effect in Propositions \ref{catalan} and \ref{linialexp}.

\begin{example}\mylabel{X:intcf}
The \emph{integral chromatic function} $\chiZ_\Phi(q)$ (from \cite{SOA}) is the number of proper colorations of $\Phi$ by colors in the set $[q]$, \emph{proper} meaning subject to the conditions given by the gains of the edges.  
This function is a weak chromatic function of integral gain graphs but it is not invariant under switching, so it is not a weak chromatic invariant.  
It is loop independent, because the color of a vertex is never constrained by a loop with nonzero additive gain.

That the integral chromatic function has neutral-loop nullity is obvious from the definition.  To show it satisfies neutral deletion-contraction, consider a proper coloration of $\Phi \setm e$, where $e$ is a neutral link, using colors in $[q]$.  If the endpoints of $e$ have different colors, we have a proper coloration of $\Phi$; if they have the same color, we have a proper coloration of $\Phi/e$.  (This argument is standard in graph coloring, corresponding to the fact that the neutral subgraph acts like an ordinary graph.)  

The reasoning fails if $e$ has non-identity gain, and switching really does change $\chiZ_\Phi(q)$.  Consider $\Phi$ with two vertices $1$ and $2$ and one edge $e$ of nonnegative gain $g \in [q]$ in the orientation from $1$ to $2$.  All such gain graphs are switching equivalent.  The rule for a proper coloration $\kappa$ is that $\kappa_2 \neq \kappa_1+g$.  Of all the $q^2$ colorations, the number excluded by this requirement is $q-g$ (or $0$ if $q-g<0$).  Assuming $0 \leq g \leq q$, $\chiZ_\Phi(q) = q(q-1) + g$, obviously not a switching invariant.
\end{example}

\begin{example}\mylabel{X:modcf}
The \emph{modular chromatic function} $\chim_\Phi(q)$ (also from \cite{SOA}) is the number of proper colorations of the vertices by colors in $\bbZ_q$.  

The remarks at the end of \cite[Section 6]{SOA} imply that $\chim_\Phi(q)$ is a weak chromatic invariant.  The idea is that $\chim_\Phi(q) = \chi^*_{\Phi \pmod q}(q)$ where $\Phi \pmod q$ is $\Phi$ with gains modulo $q$.  Take integral gain graphs $\Phi$ and $\Phi'$ such that $\Phi'$ is isomorphic to some switching of $\Phi$.  Then the same is true for $\Phi \pmod q$ and $\Phi' \pmod q$ with switching modulo $q$.  Since $\chi^*_{\Phi \pmod q}(q)$ for fixed $q$ is a weak chromatic invariant of gain graphs with gains in $\bbZ_q$, $\chim_\Phi(q)$ is a weak chromatic invariant of integral gain graphs.

The modular chromatic function is loop independent, for the same reason as is the integral chromatic function.
\end{example}

The modular chromatic function is not too different from the zero-free chromatic polynomial.  Write 
$$
{\max}_{\odot}(\Phi) := \text{ the maximum gain of any circle in } \Phi.
$$

\begin{lem}[{see \cite[Section 11.4, p.\ 339]{PDS}}]\mylabel{modchromatic}  
The modular chromatic function of an integral gain graph $\Phi$ is given by
\begin{equation*}
\chim_\Phi(q) = \chi^*_\Phi(q)  \qquad\text{ for integers } q > {\max}_{\odot}(\Phi),
\end{equation*}
but equality fails in general for $q = {\max}_{\odot}(\Phi)$.
\end{lem}

\begin{proof} 
If we take the integral gains modulo $q > {\max}_{\odot}(\Phi)$, we do not change the balanced circles, because no nonzero circle gain is big enough to be reduced to $0$.  
By Equation \eqref{E:tcpoly} with $z=0$ so the sum may be restricted to balanced edge sets $S$, we do not change the zero-free chromatic polynomial.  Proper colorations in $\bbZ_q$ with modular gains are proper modular colorations.

If $q = {\max}_{\odot}(\Phi)$, at least one unbalanced circle becomes balanced so we do change the list of balanced edge sets and we cannot expect equality.  (We have not tried to decide whether equality is possible at all.) 
\end{proof}

This lemma, though disguised by talk about finite fields and the Critical Theorem, is fundamentally the same method used by Athanasiadis in most of his examples in \cite{Athan}; see \cite[Section 6]{SOA}.

\medskip

The \emph{affinographic hyperplane arrangement} that corresponds to an integral gain graph $\Phi$ is the set $\cA$ of all hyperplanes in $\bbR^n$ whose equations have the form $x_j=x_i+g$ for edges $ge_{ij}$ in $\Phi$.  See \cite[Section IV.4]{BG} or \cite{SOA} for more detail about this connection.  A most important point is that the characteristic polynomial of this arrangement, $p_\cA(q)$, equals the zero-free chromatic polynomial $\chi^*_\Phi(q)$, by \cite[Theorem III.5.2 and Corollary IV.4.5]{BG}.  Examples include the well known Shi, Linial, and Catalan arrangements, which we will define.  (In these definitions, $\bbZ$ could be replaced by any ordered abelian group, or a subgroup of the additive group of any field; for instance, the additive real numbers.)

\section{Catalan arrangements and their graphs}

We will now apply the preceding results to obtain relations between the special gain graphs corresponding to the Shi, Linial, and Catalan arrangements.  We begin with the last.

Let $C_n = \{0,\pm1\}K_n$, the complete graph $K_n$ (on vertex set $[n]$) with gains $-1$, $0$, and $1$ on every edge $ij$; we call this the \emph{Catalan graph} of order $n$, because the corresponding hyperplane arrangement is known as the Catalan arrangement, $\cC_n$.  
Let $C_n' = \{\pm1\}K_n$, the complete graph $K_n$ with gains $-1$ and $1$ on every edge $ij$; we call this the \emph{hollow Catalan graph}.  

Let $c(n,j)$ be the number of permutations of $[n]$ with $j$ cycles.  The \emph{Stirling number of the first kind} is $s(n,j) = (-1)^{n-j}c(n,j)$.  
The \emph{Stirling number of the second kind}, $S(n,j)$, is the number of partitions of $[n]$ into $j$ blocks.

\begin{prop}\mylabel{catalan}
Let $F$ be a weak chromatic function of integral gain graphs with the property of loop independence.  
Between the Catalan and hollow Catalan graphs we have the two relations
\[
F(C_n') = \sum_{j=1}^n S(n,j) F(C_j)
\]
and
\[
F(C_n) = \sum_{j=1}^n s(n,j) F(C_j') .
\]
\end{prop}

\begin{proof}
The proof begins with Corollary \ref{complete}, which gives the identities
\begin{align*}
F(C_n) = \sum_{\pi\in\Pi_n}  \mu(0,\pi) F(C_n'/\pi) , \\
F(C_n') = \sum_{\pi\in\Pi_n}  F((C_n'/\pi)^0) .
\end{align*}
By the hypotheses on $F$, we can simplify the contractions.  Contraction by a partition introduces no new gains, but only loops and multiple edges with the same gains.  Multiple edges simplify without changing $F$ because $F$ is a weak chromatic function.  The loops can be deleted because $F$ is loop independent.  Therefore, $F(C_n'/\pi) = F(C_{|\pi|}')$ and $F((C_n'/\pi)^0) = F(C_{|\pi|})$.  It follows that 
\[
F(C_n) = \sum_{\pi\in\Pi_n}  \mu(0,\pi) F(C_{|\pi|}') = \sum_{j=1}^n s(n,j) F(C_j')
\]  
because $s(n,j) = \sum_{\pi\in\Pi_n: |\pi|=j} \mu(0,\pi),$ and
\[
F(C_n') = \sum_{\pi\in\Pi_n}  F(C_{|\pi|}) = \sum_{j=1}^n S(n,j) F(C_j).
\qquad\qedhere
\]
\end{proof}

Let $r_n$ be the number of regions in the Catalan arrangement $\cC_n$ and let $r'_n$ be the number of regions  of the arrangement corresponding to the hollow Catalan graph $C_j'$ (which Stanley calls the \emph{semiorder arrangement}).

\begin{cor}\mylabel{catalanregions}
$r_n = \sum_j c(n,j) r'_j$.
\end{cor}

\begin{proof}
The arrangement $\cH[\Phi]$ corresponding to an additive real gain graph of order $n$ has $(-1)^n\chi^*_\Phi(-1)$ regions \cite[Corollary IV.4.5(b)]{BG}.  Also, $(-1)^{n-j} s(n,j)$ is the number of $i$-cycle permutations.  
Applying the weak chromatic invariant $\chi^*$ to the second equation of Proposition \ref{catalan} and evaluating at $q=-1$, we get
$\chi^*_{C_n}(-1) = \sum_j s(n,j) \chi^*_{C'_j}(-1).$  
Since $s(n,j) = (-1)^{n-j} c(n,j)$, the equation follows.
\end{proof}

The integral, modular, and zero-free chromatic functions of $C_n$ are very simple to obtain.  The gains $0$ correspond to the condition that the colors of the vertices are all different.  The gains $-1$ and $1$ correspond to the condition that the colors of two vertices are never consecutive.  
Thus, the modular chromatic function of $C_n$ counts injections $f:V\rightarrow \bbZ_q$ such that no two values of $f$ are consecutive.  We repeat the well known evaluation \cite{Kaplansky}.  
If we shift $f$ so that $f(v_1)=0$ (thereby collapsing $q$ different injections together) and delete the successor of each value, we have an injection $\bar f:V\to \bbZ_{q-n}$ such that $\bar f(v_1)=0$, or equivalently, an arbitrary injection $\bar f': V \setm \{ v_1 \} \rightarrow \{2,\ldots, q-n\}$.  There are $(q-n-1)_{n-1}$ of these.  It follows that 
\begin{align*}
\chim_{C_n}(q) &= q (q-n-1)_{n-1} &\text{ for integers } q > n  
\intertext{(and it obviously equals $0$ for $q < 2n)$.  Similarly, the integral chromatic function is}
\chiZ_{C_n}(q) &= (q-n)_n &\text{ for integers } q \geq n
\intertext{(and $0$ for $q<2n$).  The zero-free chromatic polynomial is}
\chi^*_{C_n}(q) &= q (q-n-1)_{n-1} 
\end{align*}
by Lemma \ref{modchromatic} (or see \cite[Equation (11.1)]{PDS}, where $\chi^*_{C_n}(q)$ is called $\chib_{[-1,1]K_n}(q)$).

To obtain the various chromatic functions of the hollow Catalan graphs $C_n'$ directly is not as easy, but they follow from Proposition \ref{catalan}.  
We observed in Example \ref{X:intcf} that the integral chromatic function is a weak chromatic function and loop independent; therefore, 
\begin{align*}
\chiZ_{C_n'}(q) &= \sum_{j=1}^n S(n,j) (q-j)_j &\text{ for } q \geq n . \\
\intertext{Since the zero-free chromatic polynomial is a chromatic invariant with loop independence, } 
\chi^*_{C_n'}(q) &= q \sum_{j=1}^n S(n,j) (q-j-1)_{j-1} .
\intertext{Then the modular chromatic function follows by Lemma \ref{modchromatic} and the fact that ${\max}_{\odot}(C_n) = n$: }
\chim_{C_n'}(q) &= q \sum_{j=1}^n S(n,j) (q-j-1)_{j-1} &\text{ for } q > n .
\end{align*}
%

\section{Arrangements between Shi and Catalan}

The \emph{Shi graph} of order $n$ is $S_n = \{0,1\}\vec K_n$, i.e., the complete graph $K_n$ with gains $0$ and $1$ on every oriented edge $ij$ with $i<j$.  To have simple notation we take vertex set $[n]$ and we write all edges $ij$ with the assumption that $i<j$.

Let $G$ be a spanning subgraph of $K_n$, that is, it has all $n$ vertices; and define $SC(G)$ to be the gain graph $S_n \cup \{-1\}\vec G$, which consists of the complete graph $K_n$ with gains $0$ and $1$ on every edge $ij$, and also gain $-1$ if $ij \in E(G)$.  We call $SC(G)$ a graph \emph{between Shi and Catalan}.  If $G$ is edgeless we have the Shi graph $S_n$ and if $G$ is complete we have the Catalan graph $C_n$. We compute chromatic functions of these graphs between Shi and Catalan. 

Let us start with the case where $G$ has a unique edge $e_0=i_0j_0$ and compute the modular chromatic function for large $q$ (that is, the zero-free chromatic polynomial, by Lemma \ref{modchromatic}).  The gains $0$ correspond to the condition that the colors must be all different.  The gains $1$ on every edge correspond to the condition that if the color of $i$ immediately follows the color of $j$ then $i<j$.  Finally, the gain $-1$ on the edge $e_0$ implies that the color of $i_0$ cannot immediately follow the color of $j_0$.  
The number of modular colorations of our graph can be deduced by taking out of the list of proper $q$-colorations of $S_n$ the ones for which the colors of $i_0$ and $j_0$ are consecutive (in decreasing order).  The number of these equals $q/(q-1)$ times the number of proper  $(q-1)$-colorations of $S_{n-1}$, since we can simply remove $j_0$ and its color;  more precisely, we normalize the colorations so that $i_0$ has color $0$, thus $j_0$ has color $q-1$, and then convert by removing $j_0$ and $q-1$; $S_n$ becomes $S_{n-1}$ and $\bbZ_q$ becomes $\bbZ_{q-1}$.  
We conclude from the known value $\chi^*_{S_n}(q) =q(q-n)^{n-1}$ (\cite{Athan, Headley}; we reprove this soon) that the modular chromatic function of $SC(G)$ is $q[(q-n)^{n-1} - (q-n-1)^{n-1}]$.

This small example can make one feel the complexity of the computation in the case of a general graph $G$.  Nevertheless one can produce formulas.  

Let $G$ be a graph on vertex set $[n]$.  A \emph{descending path} in $G$ is a path $i_1i_2\cdots i_l$ such that $i_1>i_2>\cdots>i_l$.  Let $p_r(G)$ be the number of ways to partition the vertex set $[n]$ into $r$ blocks, each of which is the vertex set of a descending path.  We call such a partition a \emph{descending path partition} of $G$.

\begin{prop}\mylabel{L:SC}
Let $G$ be a simple graph with vertex set $[n]$.  The integral and modular chromatic functions and the zero-free chromatic polynomial of $SC(G)$ are given by
\begin{align*}
\chiZ_{SC(G)}(q) &= \sum_{r=1}^{n} p_r(G^c) (q-n+1)_r &\text{ for } q \geq n-1 , \\
\chim_{SC(G)}(q) &= q \sum_{r=1}^{n} p_r(G^c) (q-n-1)_{r-1} &\text{ for } q > n , \\
\chi^*_{SC(G)}(q) &= q \sum_{r=1}^{n} p_r(G^c) (q-n-1)_{r-1} .
\end{align*}
For the Shi graph in particular, 
\begin{align*}
\chiZ_{S_n}(q) &= (q-n+1)^n &\text{ for } q \geq n-1 , \\
\chim_{S_n}(q) &= q (q-n)^{n-1} &\text{ for } q > n , \\
\chi^*_{S_n}(q) &= q (q-n)^{n-1} .
\end{align*}
\end{prop}

As we mentioned, the formula for $\chi^*_{S_n}(q)$ is already known from Athanasiadis', Head\-ley's, and Postnikov--Stanley's several computations of the characteristic polynomial of the Shi arrangement (\cite[Theorem 3.3]{Athan}, \cite{Headley}, and \cite[Example 9.10.1]{PS}), since the two polynomials are equal, as we remarked in Section \ref{igg}.

\begin{proof} 
We count proper colorations, extending the method Athanasiadis used to compute $\chi_{S_n}^{mod}(q)$ (\cite{Athan}, as reinterpreted in \cite[Section 6]{SOA}). 

Our first remarks apply both to integral and modular coloring.  The gains on the edges correspond to relations on the colors of the vertices.  
The 0-edges prevent coloring two different vertices with the same color.
The 1-edges imply that if two vertices are colored with consecutive colors then the larger vertex has the
first color.  This gives a nice way to describe permissible colorations in $q$ colors.  

Since all colors used are different we know that there are exactly $q-n$ colors not used (we must assume that $n\leq q$).
Imagine these colors lined up in circular or linear order, depending on whether we are evaluating $\chim$ or $\chiZ$.  We now need to arrange the vertices in the spaces between these unused colors.   When we place some vertices in the same space they are in descending order, so their places are compelled by their labels.  
Therefore, all we have to do for the Shi graph, where there are no $-1$-edges, is to assign each of $n$ vertices to a space between the $q-n$ unused vertices.  This is the classical problem of placing labelled objects into labelled boxes.  There are $q-n$ boxes in the modular case and $q-n+1$ in the integral case.  In the modular case, we assign vertex $1$ to box $0$; the other $n-1$ vertices can be placed arbitrarily.  
Then after inserting the vertices we have a circular permutation of $q$ objects, which is isomorphic to $\bbZ_q$ in $q$ ways; this accounts for the extra factor of $q$ in the modular Shi formula.

For a gain graph between Shi and Catalan, the edges with gain $-1$ correspond to vertices $i<j$ that can be in the same box only if they are not consecutive amongst the vertices in the box; i.e., there must be present in the box at least one vertex $h$ satisfying $i<h<j$.  
More precisely, suppose the vertices in the box, in descending order, are $j_1, j_2, \ldots, j_l$.  Then no consecutive pair can be adjacent in $G$, or, to put it differently, $j_1j_2 \cdots j_l$ must be a path in $G^c$.

Now we count.  We first look at modular coloring.  Consider the colors not used to be null symbols labelled by $\bbZ_{q-n}$, i.e., these colors are cyclically ordered.  
To get the number of proper colorations, we choose a partition of $[n]$ into the vertex sets of $r$ descending paths, and then we place the $r$ paths, each one with its vertex set in descending order, into the $q-n$ spaces between the nulls.  Due to the cyclic symmetry we can fix the space before $0\in \bbZ_{q-n}$ to be the one where we put the path that contains vertex 1.  There are $(q-n-1)_{r-1}$ ways to place the other $r-1$ paths.  
Now we have a cyclic arrangement of $q$ objects, vertices and nulls.  This set is isomorphic to $\bbZ_q$ in $q$ different ways, each of which gives a different proper coloration of $SC(G)$.  We get for the modular chromatic polynomial
\begin{equation*}
q \sum_\cP (q-n-1)_{r-1},
\end{equation*}
summed over all descending path partitions $\cP$ of $G^c$, where $r$ is the number of paths in $\cP$.  Our description is meaningful so long as $q-n-1\geq 0$, since the largest possible number $r$ is $n$.

For integral coloring the technique is similar.  The nulls are linearly ordered, isomorphic to $[q-n]$, and there are $q-n+1$ boxes, i.e., spaces between and around them.  We get for the integral chromatic polynomial
\begin{equation*}
\sum_\cP (q-n+1)_r
\end{equation*}
because the $r$ paths can be placed in any distinct boxes. The computation applies as long as $q-n+1 \geq 0$.

We derived the Shi formulas by a direct calculation but it is easy to deduce them from the general descending-path formulas.  Since $G^c = K_n$, the descending-path-partition number $p_r(G^c)$ is just the number of partitions with $r$ blocks, that is, $S(n,r)$.  Then one can collapse the sums; e.g., for the zero-free chromatic polynomial, 
\[
\sum_{r=1}^{n} p_r(G^c) (q-n-1)_{r-1} = (q-n)\inv \sum_{r=1}^{n} S(n,r) (q-n)_{r} = (q-n)^{n-1} .	
\qquad\qedhere
\]
\end{proof}

When $G^c$ is the comparability graph $\Comp(P)$ of a partial ordering of $[n]$ that is compatible with the natural total ordering, i.e., such that $i <_P j$ implies $i < j$ (in $\bbZ$), a descending path is a chain, so $p_r(G^c)$ is the number of ways to partition the set $P$ into $r$ chains.

\medskip

We will now limit ourselves to the special case where $G$ is a graph of order $k$ constructed from a partition of $[n]$.  
Let $\pi$ partition $[n]$ into $k$ blocks $X_1,\ldots,X_k$, with the notation chosen so that $a_i :=\min(X_1)<a_2:=\min(X_2)<\cdots<a_k:=\min(X_k).$  Thus, $X_1$ contains $1$, $X_2$ contains the smallest element not in $X_1$, and so on.  
The blocks are naturally partially ordered by letting $X_i<X_j$ if $c<d$ for every $c\in X_i$ and $d\in
X_j$; in other words, if $b_i:=\max X_i < a_j=\min X_j$.  This partial ordering induces a partial order $P_\pi$ on $[k]$.
We say $X_i$ and $X_j$ \emph{overlap} if neither $X_i< X_j$ nor $X_i>X_j$.  

Let $\G_\pi$ be the interval graph of the intervals $[a_i,b_i]$ for $i\in[k]$.  
(See \cite{Golumbic} for the many interesting properties of interval graphs.)  
Then $\G_\pi$ has an edge $ij$ just when $X_i$ overlaps $X_j$, so its complement is the comparability graph $\Comp(P_\pi)$.  
The integral and modular chromatic functions of the gain graph $SC(\G_\pi)$ for the partition $\pi$ can be obtained directly in terms of $\G_\pi$.  
Let $d_i=d_i(\pi)$ be the \emph{lower degree} of $i$ in $\G_\pi$, i.e., the number of blocks $X_j$ overlapping $X_i$ and having $j<i$; note that $d_1=0$.

\begin{thm}\mylabel{SC}
Let $\pi$ be a partition of $[n]$ into $k$ blocks $X_1, \ldots, X_k$, and let $SC(\G_\pi)$ be the corresponding gain graph (of order $k$) between Shi and Catalan.  The integral and modular chromatic functions and the zero-free chromatic polynomial of $SC(\G_\pi)$ are:
\begin{align*}
\chiZ_{SC(\G_\pi)}(q) &= (q-k+1)\prod_{i=2}^k (q-k+1-d_i) &\text{for } q \geq k-1+\max d_i, \\
\chim_{SC(\G_\pi)}(q) &= q\prod_{i=2}^k (q-k-d_i) &\text{for } q \geq k+\max d_i, \\
\chi^*_{SC(\G_\pi)}(q) &= q\prod_{i=2}^k (q-k-d_i) .
\end{align*}
\end{thm}

\begin{proof}[Direct Proof]
Again we extend the method of Athanasiadis \cite{Athan} used to compute $\chi^*_{S_n}(q)$ \cite{Athan}, but slightly diferently from before.

We first look at modular coloring.  To get the number of proper colorations, we choose the colors of the vertices in increasing order.  To color the vertex $X_1$, we have $q$ choices.  For $X_2$ we have $q-k$ choices if $X_1$ and $X_2$ do not overlap (which corresponds to the presence of a $-1$ edge) and only $q-k-1$ choices otherwise.  We go on, and when coloring $X_i$ we have $q-k-d_i$ choices (we use the fact that two blocks overlapping with $X_i$ must overlap each other).  We get for the modular chromatic polynomial 
\[
q \prod_{i=2}^k (q-k-d_i).
\]
The lower bound on $q$ arises from the fact that $q-k-d_i$ must never be negative if the reasoning is to hold good.

For integral coloring the technique is similar.  To color $X_1$ we have $q-k+1$ choices, which is the number of boxes.  To color $X_i$ we have $q-k+1-d_i$ choices, the number of boxes minus the number of forbidden boxes.  We get for the integral chromatic polynomial
\[
(q-k+1) \prod_{i=2}^k (q-k+1-d_i) = \prod_{i=1}^k (q-k+1-d_i) .
\qquad\qedhere
\]
\end{proof}

\begin{proof}[Deduction from Proposition \ref{L:SC}]
A \emph{simplicial vertex ordering} in a graph $G$ is a numbering of the vertices by $1,2, \ldots, k$ such that, for each $r$, in the subgraph $G_r$ induced by the vertices $1,\ldots, r$ the neighborhood of $r$ is a clique (see, for instance, \cite{Golumbic}).  In $G=\G_\pi$ it is easy to see that the natural ordering of $[n]$ is a simplicial vertex ordering and the number of neighbors of $r$ in $G_r$ is the lower degree $d_r$.  A descending path is a chain in $P_\pi$.

We apply Proposition \ref{L:SC} inductively, leaving the easy case $k=1$ to the reader.  Suppose it is true for $P_\pi \setm k$.  A chain decomposition of $P_\pi$ is obtained by taking an $r$-chain decomposition of $P_\pi \setm k$ and adjoining $k$ in either of two ways: we can add a new chain $\{ k \}$, or we can add $k$ to an existing chain $i_1 > \cdots > i_l$, necessarily at the top.  This is possible if and only if $k > i_1$ in $P_\pi$.  Since the non-neighbors of $k$ form a clique, they are mutually incomparable.  
Thus, each one is in a separate chain.  Each must be the top element of its chain because otherwise the top element would be $< k$ and the non-neighbor would also be $< k$ by transitivity.  It follows that the number of chains to which $k$ can be added is $r-d_k$.  We conclude that 
\begin{equation*}
p_r(\G_\pi^c) = p_{r-1}(\G_\pi^c \setm k) + (r-d_k)p_r(\G_\pi^c \setm k).
\end{equation*}

Now we compute the value of the right-hand side of an expression in Proposition \ref{L:SC}; we do the integral chromatic function, the others being similar. From the lemma,
\begin{align*}
\chiZ_{SC(\G_\pi)}(q) &= \sum_{r=1}^k p_r(\G_\pi^c) (q-k+1)_r \\
  &= \sum_{r=2}^k p_{r-1}(\G_\pi^c \setm k) (q-k+1)_{r-1} (q-k-[r-2]) \\
  &\quad + \sum_{r=1}^{k-1} (r-d_k) p_r(\G_\pi^c \setm k) (q-k+1)_r 
\intertext{because $p_0(\G_\pi^c \setm k) = p_k(\G_\pi^c \setm k) = 0$,}
  &= (q-k+1-d_k) \sum_{r=1}^{k-1} p_r(\G_\pi^c \setm k) (q-k+1)_r \\
  &= (q-k+1-d_k) \chiZ_{SC(\G_\pi \setm k)}(q-1) \\
  &= (q-k+1-d_k) \prod_{i=1}^{k-1} ([q-1] - [k-1] + 1 - d_i) \\
\end{align*}
by induction.  This is the desired formula.
\end{proof}

\section{Linial arrangements}

The \emph{Linial graph} of order $n$ is $L_n = 1\vec K_n$, i.e., the complete graph $K_n$ with gains 1 on every oriented edge $ij$ (that is, with $i<j$).  We found the integral chromatic function of the Linial graph in \cite{SOA}, but here we have a new and different formula that also gives the modular and zero-free chromatic functions.  A corollary is a new formula for the characteristic polynomial of the Linial hyperplane arrangement, different from (though not as simple as) that of Athanasiadis \cite[Theorem 4.2]{Athan} (differently proved in \cite[Example 9.10.3]{PS}).

We expand the Linial gain graph in gain graphs between Shi and Catalan.  
The chromatic functions we are calculating are invariant under simplification that does not remove neutral loops and are not affected by non-neutral loops.

\begin{cor}\mylabel{linialexp}
Let $F$ be a weak chromatic function of integral gain graphs which is loop independent.  Then
\[
F(L_n) = \sum_{\pi\in \Pi_n} F(SC(\G_\pi)) .
\]
\end{cor}

\begin{proof}
A straightforward application of Corollary \ref{complete} shows that
\[
L_n = \sum_{S\in \cF(\G_0)} S_n/S.
\]
The flat $S$ corresponds to a partition $\pi$; it is the union of neutral cliques on the blocks of $\pi$.  The edges that remain after contraction are the neutral edges, which connect all the blocks of $\pi$ forming a complete neutral subgraph of $S_n/S$, and the contractions of $1$-edges.  
Number the blocks by least element and partially order as in the previous section.  If $i<j$, then there is a $1$-edge from $X_i$ to $X_j$ in the contraction. 
If furthermore $X_i < X_j$ in $P_\pi$, there is a $(-1)$-edge in the rising direction.  
We can simplify multiple edges with the same gain, by definition of $F$.  
We can ignore neutral loops because contracting $S$ leaves none.  Contraction makes a $1$-loop at each contraction vertex $X\in\pi$ that has a $1$-edge, but these do not affect the value of $F$.  
Therefore, $F(S_n/S) = F(SC(\G_\pi))$.
\end{proof}

We get the next result from Theorem \ref{SC} and Corollary \ref{linialexp}.

\begin{thm}\mylabel{linialpolys}
The integral and modular chromatic functions of the Linial gain graph satisfy
\begin{align*}
\chiZ_{L_n}(q) &= \sum_{\pi\in\Pi_n}  (q-k+1) \prod^k_{i=2} (q-k+1-d_i) &\text{ for } q \geq n-1, \\
\chim_{L_n}(q) &= q \sum_{\pi\in\Pi_n}  \prod^k_{i=2} (q-k-d_i) &\text{ for } q \geq n, \\
\chi^*_{L_n}(q) &= q \sum_{\pi\in\Pi_n}  \prod^k_{i=2} (q-k-d_i) ,
\end{align*}
respectively, where $k = |\pi|$ and $d_i$ is the lower degree $d_i(\pi)$.
\end{thm}

\begin{proof}
We get the lower bound on $q$ in the modular polynomial from Lemma \ref{modchromatic}, since the circle with maximum gain is $12 \cdots n1$, using $1$-edges in the upward direction and the $0$-edge $n1$.  The gain is $(n-1)1+0$.

In the integral case the lower bound follows from the obvious necessity that $q-|\pi|+1\geq 0$ for every partition.
\end{proof}

The zero-free chromatic polynomial, by a remark in Section \ref{igg}, is also the characteristic polynomial of the Linial arrangement $\cL_n$; thus,
\[
p_{\cL_n}(q) = q \sum_{\pi\in\Pi_n}  \prod^k_{i=2} (q-k-d_i) .
\]
This new formula contrasts with that of Athanasiadis:  
\[
p_{\cL_n}(q) = \frac{q}{2} \sum_{j=0}^n \binom{n}{j} \left(\frac{q-j}{2}\right)^{n-1} .
\]
Interesting enumerative conclusions might follow from the equality of these two expressions for the Linial polynomial, but that is too complicated to pursue here.  At any rate, our formula has an interesting combinatorial aspect.

\begin{example}\mylabel{X:linial}
Let $n=6$ and $\pi=\{13,25,46\}$, so $k=3$.  
Then $\G_\pi$ has edges $12$ and $23$ but not $13$.  We have $d_1=0, d_2=d_3=1$.  
The formula for integral colorations is $(q-2)(q-3)^2$, which gives two proper colorations for $q=4$.  
The colorations are given by the sequences $31o2$ and $2o31$, where the color corresponds to the position in the sequence, the numbers denote the vertices, and $o$ denotes an unused color.  
Thus in the first sequence vertex $3$ is colored $1$, vertex $1$ has color $2$, color $3$ is not used, etc.  The coloration is built up by taking $q-n=1$ unused color, forming the sequence $o$, which makes two boxes to place descending paths in $\G_pi^c$, and placing the descending paths $31$ and $2$ in the first and second boxes, respectively.  
\end{example}

We see that different partitions $\pi$ can give the same terms because the term of a partition depends only on the lower degrees.  We would like to know, for a given nondecreasing sequence $D = (d_1=0,d_2,\ldots, d_k)$, the number $N_1(D)$ of partitions of $[n]$ which correspond to this sequence.  
Knowing these numbers will give us formulas like
\[
\chim_{L_n}(q) = q \sum_{D} N_1(D) \prod^k_{i=2}(q-k-d_i).
\]
(In fact, although each term of this sum depends only on the elements of the sequence and not on their order, we would be pleased to know also the number $N_2(D)$ of naturally ordered partitions of $[n]$ which correspond to an arbitrary ordered sequence $D$.)

Define $D_1(\pi)$ to be the \emph{increasing lower degree sequence} of $\G_\pi$, which is the sequence of lower degrees written in non-decreasing order; and let $D_2(\pi)$ be the sequence of lower degrees in \emph{vertex order}, i.e., where $d_i=d_i(\pi)=$ the lower degree of $X_i$.  To recognize these sequences is not difficult.   
Call an \emph{ascent} of a sequence $D=(d_1,d_2,\ldots,d_k)$ any position $i \in [k-1]$ such that $d_{i+1}>d_i$.  The \emph{ascent set} is $A(D) := \{ i \in [k-1] : d_{i+1}>d_i \}.$

\begin{prop}\mylabel{lds}
A sequence $D=(d_1,\ldots,d_k)$ is the vertex-order lower degree sequence $D_2(\pi)$ of the overlap graph $\G_\pi$ for some $\pi\in\Pi_n$ if and only if 
\begin{gather*}
d_1=0\leq d_2, \ldots, d_k, \\
d_{i+1} \leq d_i+1 \text{ for every } i\in[k-1], \text{ and} \\
n \ge k + \text{the number of ascents of } D .
\end{gather*}
\end{prop}

\begin{proof}
In the proof let $\sim$ denote adjacency in $\G_\pi$.

First we prove the three conditions are necessary.  Let $\pi$ be any partition of $[n]$ into $k$ blocks; let $D = D_2(\pi)$.  It is clear that $D$ has the first property.  
For the second, let $j<i \leq k$.  If $X_j$ overlaps $X_{i+1}$ it also overlaps $X_i$; thus, if $j \sim i+1$ then $j \sim i$.  Consequently, $d_{i+1} \leq d_i+1$ and equality holds only if $a_{i+1} < b_i$.  The latter implies $b_i>a_i$.  We conclude that, in $D$, 
\begin{equation}\mylabel{longblock}
d_i<d_{i+1} \implies |X_i|>1.
\end{equation}
Consequently, every ascent of $D$ requires a block of size at least $2$; this gives the lower bound on $n$.

For sufficiency of the three conditions we assume a sequence $D$ has the three properties of the proposition.  We construct a partition $\pi \in \Pi_{n_D}$ that realizes $D$, where $n_D := k + |A(D)|$.  (For $n>n_D$ we simply add $n_D+1,\ldots,n$ to the block of $\pi$ that contains $n_D$.)  

Define $d_0,\ d_{k+1} := 0$ and put $[i,j] := \{i,i+1,\ldots,j\}$ if $i\leq j$.  
Now, let $X_i = \{a_i,b_i\}$ where
\begin{align*}
m_i &:= \min\{ j \geq i : d_{j+1} \leq d_i \}, \\
t_i &:=  |A(D) \cap [i]|, \\
a_i &:= i + t_{i-1} - d_i , \\
b_i &:= a_i + (m_i-i) + |A(D) \cap [i, m_i]| \\
&\;= m_i + t_{m_i} - d_i .
\end{align*}
Finally, let $\pi := \{X_1,\ldots, X_k\}$.  
Regardless of whether $\pi$ is a partition or not, it has an overlap graph $\G_\pi$ which is the interval graph of the intervals $[a_i,b_i]$, and if all the $a_i$ are distinct there is a natural ordering of the vertices of $\G_\pi$.

\begin{lem}\mylabel{ldslemma}
The class $\pi$ constructed from $D$ satisfies:
\begin{enumerate}
\item[(i)] $\pi$ is a partition of $[n_D]$.
\item[(ii)] $X_k$ is the singleton block $\{a_k = n_D - d_k\}$.
\item[(iii)] The upper neighborhood of $i$ in $\G_\pi$ is $[i+1,m_i]$.
\item[(iv)] The lower neighbors of $i$ in $\G_\pi$ are the vertices $i_\delta := \max\{ j \in [i-1] : d_j = \delta\}$ for $\delta = 0,1,\ldots,d_i-1$.
\item[(v)] $\pi$ has vertex-order lower degree sequence $D_2(\pi)=D$.
\end{enumerate}
\end{lem}

\begin{proof}
Note that the $t_i$ are weakly increasing.  Part (ii) is obvious because $m_k=k$.

We show that $a_1<a_2<\cdots<a_k$.  This will imply that, if we linearly order the $X_i$ by least element, the ordering is subscript order.  We calculate the increment:
\[
a_{i+1}-a_i = 1 + (t_i - t_{i-1}) + (d_i-d_{i+1}) = 1 + u_i + (d_i-d_{i+1}) .
\]
If $i$ is not an ascent, then $d_i-d_{i+1}\geq0$ and $u_i=0$, so $a_{i+1}-a_i > 0$.  If $i$ is an ascent, then $d_i-d_{i+1} = -1$ and $u_i=1$; again $a_{i+1}-a_i > 0$.

Now we prove (iii).  The first step is an equivalence, in which we assume $i<j$:
\[
i \sim j \iff a_j \leq b_i \iff j + t_{j-1} - d_j \leq m_i + t_{m_i} - d_i .\
\]
If $j \leq m_i$, then $d_j > d_i$ so $j + t_{j-1} - d_j < m_i + t_{m_i} - d_i$.
If $j>m_i$, then $t_{j-1} \geq t_{m_i}$.  Suppose $a_j \leq b_i$.  Then $j-m_i \leq (j-m_i) + (t_{j-1}-t_{m_i} \leq d_j - d_{m_i} \leq j - m_i.$  Thus, the inequalities are all equalities.  It follows that $d_j-d_{m_i} = j-m_i$, which makes $m_i$ an ascent, but by definition $m_i$ can never be an ascent.  Note that in both cases, $a_j = b_i$ is impossible.  Therefore, no $b_i$ and $a_j$ can be equal if $i<j$.  Neither can $b_j$ equal $a_i$, because $b_j\geq a_j > a_i$.

Recall that $n_D = k + t_k$.  Let $d_j$ be the last member of $D$ that has value $0$.  Then $t_j=k$ so $b_j=n_D$.  We prove that all other $b_i < b_j$.  If $i<j$, then $t_i < j$ so $b_i < a_j \leq b_j$.  If $i>j$, then $d_i$ is positive; therefore, $b_i = m_i + t_{m_i} - d_i < m_i + t_{m_i} \leq k + t_k = b_j$.

Since $n_D$ is the number of blocks plus the number of doubleton blocks, and we have shown that none of the $a$'s and $b$'s can be equal except when $a_i=b_i$, we conclude that $\pi$ is a partition of $[n_D]$.

It remains to prove (iv), from which (v) follows at once.  
Let $j < i$.  We know from (iii) that 
\[
j \sim i \iff m_j \geq i \iff d_{j+1}, d_{j+2}, \ldots, d_i > d_j .
\]
It is easy to see that the last property implies, and is implied by, the property that $j = i_\delta$ for some $\delta$, specifically for $\delta = d_j$.
\end{proof}

The lemma obviously implies Proposition \ref{lds}.
\end{proof}

\begin{prop}\mylabel{increasinglds}
A sequence $D=(d_1,\ldots,d_k)$ is the increasing lower degree sequence $D_1(\pi)$ of $\G_\pi$ for some partition $\pi\in\Pi_n$ if and only if it satisfies 
\begin{gather*}
0=d_1\leq d_2 \leq \cdots \leq d_k, \\
d_{i+1} \leq d_i+1 \text{ for every } i\in[k-1], \\
n \geq k+d_k.
\end{gather*}
\end{prop}

\begin{proof}
$D$ is a lower degree sequence if and only if some rearrangement of it, $D'$, satisfies the conditions of Proposition \ref{lds}.  
The first two properties follow from those of $D'$.  
Neglecting the exact value of $n$, we see from Proposition \ref{lds}  that $D$ is a lower degree sequence if and only if it is itself, without rearrangement, a vertex-order lower degree sequence.
The lower bound on $n$ follows from that fact, because the smallest possible number of ascents of any rearrangement of $D$ is $d_k = \max d_i$.
\end{proof}

\section{The total chromatic polynomial}

We did not try to evaluate the chromatic polynomial of the Catalan and other graphs, all the more the total chromatic polynomial, because the results do not appear to be very nice nor are they known to count anything interesting.  However, the total chromatic polynomial can be computed in terms of the zero-free polynomial, extending the balanced expansion 
\[
\chi_\Phi(q) = \sum_{\substack{W \subseteq V\\ W^c \text{stable}}} \chi_{\Phi\ind W}^*(q-1)
\]
from \cite[Theorem III.6.1]{BG}.  Now we develop this computation to the extent that evaluating the polynomial for the Catalan, Shi, or Linial graph becomes a mechanical exercise.

\begin{prop}\mylabel{totalexpansion}
Let $\Phi$ be a gain graph with underlying graph $\G$.  The total chromatic polynomial satisfies the expansion identity
\[
\tilde\chi_\Phi(q,z) = \sum_{W \subseteq V}  \chi_{\G\ind W^c}(z) \chi_{\Phi\ind W}^*(q-z) .
\]
\end{prop}

\begin{lem}\mylabel{totalmult}
If $\Phi_1, \ldots, \Phi_c$ are the components of $\Phi$, then 
$$
\tilde\chi_\Phi(q,z) = \tilde\chi_{\Phi_1}(q,z) \cdots \tilde\chi_{\Phi_c}(q,z).
$$
\end{lem}

\begin{proof}
This is obvious from the combinatorial definition.  The proof from the algebraic definition is standard (and easy) and is therefore omitted.
\end{proof}

\begin{proof}[Combinatorial Proof of the Proposition]
Assume $\fG$ is finite, $k, z$ are nonnegative integers, and $q=k|\fG|+z$.  We count the proper colorations of $\Phi$ by the color set $\big( [k] \times \fG \big) \cup [z]$.  We pick a vertex subset $W$ to color by $[k] \times \fG$, leaving the complement to be colored by $[z]$.  The number of ways to color $W$ properly is $\chi_{\Phi\ind W}^*(q-z)$.  The number of ways to color $W^c$ properly is $\chi_{\G\ind W^c}(z)$ because a coloration is proper if and only if no edge has the same $[z]$-color at both ends.
\end{proof}

\begin{proof}[Algebraic Proof]
The algebraic proof applies to all gain graphs.  
If $\Phi$ has no links, then the components are the single vertices with their loops, if any, for which the expansion is obviously correct.
Thus, we may apply induction on the number of links.  Assume there is a link $e$.  

We develop the right-hand side of the expansion identity.  
We split the sum according to three cases: $e \in E\ind W$, $e \in E\ind W^c$, and $e$ not in either induced edge set.  
In the first case, $\chi_{\Phi\ind W}^*(q-z) = \chi_{(\Phi\setm e)\ind W}^*(q-z) - \chi_{(\Phi/e)\ind W}^*(q-z).$  
In the second case, $\chi_{\G\ind W^c}(z) = \chi_{(\G\setm e)\ind W^c}(z) - \chi_{(\G/e)\ind W^c}(z).$  Let us write this all out.  The right-hand side equals
\begin{align*}
&\sum_{\substack{W \subseteq V\\ e \in E\ind W}}  \chi_{\G\ind W^c}(z) \chi_{(\Phi\setm e)\ind W}^*(q-z) 
 + \sum_{\substack{W \subseteq V\\ e \in E\ind W^c}}  \chi_{(\G\setm e)\ind W^c}(z) \chi_{\Phi\ind W}^*(q-z) \\
&\qquad + \sum_{\substack{W \subseteq V\\ e \notin (E\ind W) \cup (E\ind W^c)}}  \chi_{\G\ind W^c}(z) \chi_{\Phi\ind W}^*(q-z) \\
&- \bigg[  \sum_{\substack{W \subseteq V\\ e \in E\ind W}}  \chi_{\G\ind W^c}(z) \chi_{(\Phi/e)\ind W}^*(q-z)
+ \sum_{\substack{W \subseteq V\\ e \in E\ind W^c}}  \chi_{(\G/e)\ind W^c}(z) \chi_{\Phi\ind W}^*(q-z)  \bigg] \\
\end{align*}
The first three terms are precisely those in the expansion of $\tilde\chi_{\Phi\setm e}(q,z)$.  
As for the remaining two terms, let $v_e$ be the vertex of $\Phi/e$ that results from contracting $e$.  Then
\begin{align*}
&\sum_{\substack{W \subseteq V\\ e \in E\ind W}}  \chi_{\G\ind W^c}(z) \chi_{(\Phi/e)\ind W}^*(q-z)
+ \sum_{\substack{W \subseteq V\\ e \in E\ind W^c}}  \chi_{(\G/e)\ind W^c}(z) \chi_{\Phi\ind W}^*(q-z)  \\
&= \sum_{\substack{X \subseteq V(\Phi/e)\\ v_e \in E(\Phi/e)\ind X}}  \chi_{\G\ind X^c}(z) \chi_{(\Phi/e)\ind X}^*(q-z)
+ \sum_{\substack{X \subseteq V(\Phi/e)\\ v_e \in E(\Phi/e)\ind X^c}}  \chi_{(\G/e)\ind X^c}(z) \chi_{\Phi\ind X}^*(q-z) 
\end{align*}
because a set $X \subseteq V(\Phi/e)$ corresponds to a set $W = X \setm \{v_e\} \cup \{v,w\}$ if $v,w$ are the endpoints of $e$ in $\Phi$, so that $e$ is in either $E\ind W$ or $E\ind W^c$.  The last two sums are the expansion of $\tilde\chi_{\Phi/e}(q,z)$.
\end{proof}

Call $\Phi$ \emph{complete} if every two vertices are adjacent by at least one edge.  Recall that $n$ denotes the order of $\Phi$.

\begin{cor}\mylabel{totalcomplete}
If $\Phi$ is complete, let $L$ be the set of vertices that support loops.  Then the total chromatic polynomial expands with falling-factorial coefficients: 
\[
\tilde\chi_\Phi(q,z) = \sum_{L \subseteq W \subseteq V}  (z)_{n-|W|} \, \chi_{\Phi\ind W}^*(q-z) .
\]
\end{cor}

\begin{proof}
The factor $\chi_{\G\ind W^c}(z)$ equals zero if there is a loop in $W^c$ and otherwise it is the chromatic polynomial of $K_{|W^c|}$.
\end{proof}

Call $\Phi$ \emph{uniform} if all induced subgraphs of the same order are isomorphic.  Then we can write $\Phi_m$ for an induced subgraph of order $m$; $\Phi_n$ is just $\Phi$.  A uniform gain graph that is not complete has no links at all.  If it has any loop, it has the same number of loops with the same gains at every vertex.

\begin{cor}\mylabel{totaluniform}
Suppose $\Phi$ is connected and uniform and has no loops.  Then 
\[
\tilde\chi_\Phi(q,z) = n! \sum_{j=0}^n  \binom{z}{n-j} 
 \frac{1}{j!} \chi_{\Phi_j}^*(q-z) .
\]
In particular, for the chromatic polynomial we have
\[
\chi_\Phi(q) = \chi_{\Phi_n}^*(q-1) + n \chi_{\Phi_{n-1}}^*(q-z) .
\]
\end{cor}

\newcommand\XGF{\operatorname{XGF}}
If there is an infinite uniform sequence $\Phi_0, \Phi_1, \ldots$ we can write Corollary \ref{totaluniform} in terms of the exponential generating function, $\XGF(a_n) := \sum_{n=1}^\infty a_n t^n/n!$, as
\begin{equation}\mylabel{uniformgf}
\XGF(\tilde\chi_{\Phi_n}(q,z)) = (1+t)^z \XGF(\chi_{\Phi_n}^*(q-z)).
\end{equation}

Since the Catalan, Shi, and Linial graphs satisfy the hypotheses of Corollary \ref{totaluniform} and the graphs between Catalan and Shi are complete, it is a routine matter to evaluate their chromatic and total chromatic polynomials by means of Corollary \ref{totaluniform} or Equation \eqref{uniformgf}.  The intermediate graphs $SC(G)$ fall under Corollary \ref{totaluniform}.

\section{Final remarks:  Extended arrangements; weighted gain graphs; biased graphs}

There are various ways to extend the Linial and Shi arrangements to more hyperplanes; each one corresponds to an integral gain graph of the form $\{-l+1, -l+2, \ldots, m\}\vec K_n$, where $l,m\geq0$ and $m>-l$.  For the characteristic polynomial the cases $l=0, 1$ were treated by Athanasiadis \cite[Section 4]{Athan}, the general case by Postnikov and Stanley \cite[Section 9]{PS}.  It would certainly be interesting to apply our results to these arrangements.

A \emph{weighted integral gain graph} $(\Phi,h)$ is an integral gain graph $\Phi$ together with a function $h : V \to \bbZ$, called the \emph{weight} function.  By the contraction rule for weights, when contracting a neutral edge set $S$ the contracted weights are $h_S(W) := \max_{v\in W} h(v)$, where $W$ is the vertex set of a component of $S$.  
For a weighted integral gain graph $(\Phi,h)$ one can define an integral chromatic function $\chiZ_{(\Phi,h)}(q)$ similar to that of $\Phi$; indeed, our $\chiZ_\Phi(q)$ is $\chiZ_{(\Phi,0)}(q)$ (all weights equal $0$).  
If $e$ is a neutral link, then (from \cite{SOA}, reinterpreted from the equivalent rooted integral gain graphs of that paper) 
 \[
 \chiZ_{(\Phi,h)}(q) = \chiZ_{(\Phi\setm e,h)}(q) - \chiZ_{(\Phi/e,h_e)}(q) .
 \]
This shows there are interesting functions $F$ on other structures than ordinary graphs and gain graphs---in fact, on arbitrary weighted gain graphs as defined in \cite{HPT}---to which the general reductions Theorems \ref{first} and \ref{second} can apply.

Another kind of generalization is to biased graphs \cite{BG}.  The balanced circles in a gain graph form a \emph{linear subclass}, which means that if a theta subgraph contains two balanced circles, then its third circle is balanced.  A \emph{biased graph} is a pair $\Omega = (\G,\cB)$ where $\cB$ is any linear subclass of the circles of $\G$.  An edge set is called balanced when every circle belongs to $\cB$.  Thus, a biased graph is a combinatorial generalization of a gain graph.  There are no distinguished neutral edges in a biased graph, but one can choose any balanced edge set to play the same role.  Thus, a deletion-contraction formula like that of Theorem \ref{first} exists.  There are also addition-contraction results where one adds edges to a balanced subgraph, though since it is probably not true that one can extend any balanced subset to a balanced $K_n$ (this is an open question), there would be no complete generalization of Theorem \ref{second}.

\section*{Acknowledgement}

We thank the referee for astute refereeing that helped to perfect the exposition.


\end{document}